%% file: GSI25-070.tex
\def\anonymous{0}

\documentclass[runningheads]{llncs}

\usepackage[utf8]{inputenc}
\usepackage{amsmath,amssymb,amsfonts}
\usepackage{yhmath}
\usepackage{standalone}
\usepackage{subcaption}
\usepackage{bm}
\usepackage{tikz}
\usetikzlibrary{calc,angles,quotes,perspective}

\usepackage[T1]{fontenc}
\usepackage{graphicx}
\usepackage[hidelinks]{hyperref}
\usepackage{color}

\usepackage{cleveref}

\newcommand{\bbR}{\mathbb{R}}
\newcommand{\bbM}{\mathbb{M}}

\newcommand{\cG}{\mathcal{G}}
\newcommand{\cH}{\mathcal{H}}

\newcommand{\dif}{\textrm{d}}

\DeclareMathOperator{\E}{E}
\DeclareMathOperator{\SE}{SE}
\DeclareMathOperator{\SO}{SO}

\let\act\vartriangleright
\DeclareMathOperator{\stab}{stab}

\begin{document}

\title{Roto-Translation Invariant Metrics on Position-Orientation Space}

\titlerunning{\(\SE(3)\) Invariant Metrics on \(\bbM_3\)}

{\if\anonymous0
    \author{
        Gijs Bellaard\and
        Bart M. N. Smets
    }
    \authorrunning{G. Bellaard et al.}
    \institute{
        CASA \& EAISI, Eindhoven University of Technology \\
        \email{
            \{g.bellaard,b.m.n.smets\}@tue.nl
        }
    }
\else
    \author{Anonymous Author(s)}
    \authorrunning{Anonymous Author(s)}
    \institute{Anonymous Institute(s)}
\fi}

\maketitle 


\begin{abstract}

Riemannian metrics on the position-orientation space \(\bbM_3 := \bbR^3 \times S^2\) that are roto-translation group \(\SE(3)\) invariant play a key role in image analysis tasks like enhancement, denoising, and segmentation. 
These metrics enable roto-translation equivariant algorithms, with the associated Riemannian distance often used in implementation.

However, computing the Riemannian distance is costly, which makes it unsuitable in situations where constant recomputation is needed. 
We propose the \emph{mav} (minimal angular velocity) \emph{distance}, defined as the Riemannian length of a geometrically meaningful curve, as a practical alternative.

We see an application of the \emph{mav distance} in geometric deep learning. 
Namely, neural networks architectures such as PONITA, relies on geometric invariants to create their roto-translation equivariant model.
The \emph{mav distance} offers a \emph{trainable} invariant, with the parameters that determine the Riemannian metric acting as learnable weights. 

In this paper we:
1) classify and parametrize all \(\SE(3)\)-invariant metrics on \(\bbM_3\),
2) describes how to efficiently calculate the \emph{mav distance},
and 3) investigate if including the \emph{mav distance} within PONITA can positively impact its accuracy in predicting molecular properties.


\end{abstract}

\section{Introduction}

Riemannian metrics \(\cG\) on position-orientation space \(\bbM_3 := \bbR^3 \times S^2\) that are roto-translation group \(\SE(3) := \bbR^3 \rtimes \text{SO}(3) \) invariant appear in various works related to medical image analysis, such as enhancement, denoising, and segmentation \cite{duits2011left,portegies2015new,smets2023pde,smets2021total}.
Such invariant metrics \(\cG\) are used to define (nonlinear) PDEs on \(\bbM_3\) which process the data in a roto-translation equivariant manner.

As evidenced in \cite{duits2011left,portegies2015new,smets2023pde}, we are in practice interested in the induced Riemannian distance \(d_\cG : \bbM_3 \times \bbM_3 \to \bbR_{\geq 0}\) defined by \(d_{\cG}(p_1,p_2) := \inf_{\gamma \in \Gamma} L_{\cG}(\gamma)\)
where \(\Gamma\) is the set of all (piecewise) continuously differentiable curves \(\gamma : [0,1] \to \bbM_3\) between \(\gamma(0)=p_1\) and \(\gamma(1)=p_2\), and \(L_{\cG}(\gamma) := \int_0^1 \|\dot \gamma\|_{\cG} \dif t\) is the length of the curve \(\gamma\).
This is because \emph{if} one has access to the Riemannian distance, it can be used to efficiently solve the PDEs numerically \cite[eq.44-46]{smets2023pde}.

However, in general, the exact Riemannian distance on \(\bbM_3\) is expensive to determine \cite{duits2018optimal}.
In \cite{portegies2015new} it is therefore suggested to instead use, what we will call, the \emph{mav distance} \(\mu_\cG(p_1,p_2) = L_{\cG}(\eta)\),
where \(\eta : [0,1] \to \bbM_3\) is the curve \(\eta(t) = e^{M(p_1,p_2) t} \act p_1\), and \(M(p_1, p_2) \in \mathfrak{se}(3)\) the unique generator between \(p_1\) and \(e^{M(p_1,p_2)} \act p_1 = p_2\) with \underline{m}inimal \underline{a}ngular \underline{v}elocity (mav).
The action \(\act\) of \(\SE(3)\) on \(\bbM_3\), and the angular velocity of a roto-translation generator will be defined in \Cref{sec:prelim}.

The mav distance is favored over the Riemannian distance as it is more tractable: the calculation of the length of \(\eta\) is straightforward. 
Namely, one can show that \(L_\cG(\eta) = \| M(p_1, p_2) \act p_1 \|_\cG\).
The action \(\act\) of \(\mathfrak{se}(3)\) on \(\bbM_3\) will be defined in \Cref{sec:prelim}.

%

We see an application of the mav distance \(\mu_\cG\) in equivariant machine learning, specifically  in the PONITA architecture introduced by Bekkers et al. in \cite{bekkers2024fast}.
PONITA is a roto-translation group \(\SE(3)\) equivariant machine learning architecture that utilizes scalar fields on position-orientation space \(\bbM_3\).
This architecture achieves state-of-the-art results on two molecular datasets: rMD17 \cite{christensen2020role}, where the task is to predict molecular dynamics, and QM9 
\cite{ramakrishnan2014quantum,ruddigkeit2012enumeration}, where the goal is to predict chemical properties of various molecules.

As described in \cite{bekkers2024fast}, one can create \(\SE(3)\) equivariant neural networks on \(\bbM_3\) through scalar \(\SE(3)\) invariants \(\iota : \bbM_3 \times \bbM_3 \to \bbR\), that being functions with the property that \(\iota(g \act p_1, g \act p_2) = \iota(p_1, p_2)\) for all \(p_1, p_2 \in \bbM_3\) and \(g \in \SE(3)\).

We recognized that the mav distance \(\mu_\cG\) can be an interesting invariant to consider within the PONITA architecture.
Namely, all invariant metrics \(\cG\) on a homogeneous space can be parametrized through a small set of real numbers -- the metric parameters \(w_i\) -- and \(\bbM_3\) is no different.
This means that the mav distance \(\mu_{\cG}\) is in fact a \emph{trainable} invariant with the metric parameters \(w_i\) acting as learnable weights.

We therefore want to investigate if the mav distance \(\mu_\cG(p_1,p_2)\) can positively impact the accuracy of the PONITA architecture.
However, before we can do this we need to
1) classify and parametrize all \(\SE(3)\) invariant metrics \(\cG\) on \(\bbM_3\), and 2) explicitly construct the mav generator \(M(p_1,p_2)\).

\subsection{Contributions}

We classify and parametrize all \(\SE(3)\) invariant metrics \(\cG\) on \(\bbM_3\).
We explicitly construct the mav generator \(M(p_1, p_2) \in \mathfrak{se}(3)\) between two position-orientations \(p_1, p_2 \in \bbM_3\).
With these two results, we can employ the mav distance \(\mu_\cG\) as a trainable invariant in the PONITA architecture.
We examine how incorporating the mav distance \(\mu_\cG(p_1, p_2)\) into the PONITA architecture impacts its accuracy.

\subsection{Outline}

In \Cref{sec:prelim} we define our main objects of study; that being the position-orientation space \(\bbM_3\) and the roto-translation group \(\SE(3)\).
In \Cref{sec:invariant_metrics} we classify and parametrize all \(\SE(3)\) invariant metrics \(\cG\) on \(\bbM_3\).
In \Cref{sec:mav} we explicitly construct the mav generator \(M(p_1, p_2)\).
In \Cref{sec:experiments} we experimentally evaluate the mav distance \(\mu_\cG\) using the PONITA architecture.
In \Cref{sec:conclusion} we conclude the paper.

\section{Preliminaries} \label{sec:prelim}

In this section we briefly define our central concepts: position-orientation space \(\bbM_3\), the roto-translation group \(\SE(3)\), and how \(\SE(3)\) acts on \(\bbM_3\).

\begin{definition}[Position-Orientation Space] \label{def:pos_ori_space}
    The smooth manifold of position-orientations is
    \( \bbM_3 = \{ (x,n) \in \bbR^3 \times \bbR^3 \mid \|n\|=1 \}\).
    The tangent space at a point \(p = (x,n) \in \bbM_3\) is \( T_p \bbM_3 = \{ (\dot x, \dot n) \in \bbR^3 \times \bbR^3 \mid \dot n \cdot n = 0 \}\).
\end{definition}

\begin{definition}[Roto-translation Group]
    The roto-translation group is
    \(\SE(3) = \{ (t, R) \in \bbR^3 \times \bbR^{3 \times 3} \mid R^\top R = I, \det R = 1\} \).
    The group product is
    \( (t_2, R_2) \cdot (t_1, R_1) = (t_2 + R_2 t_1, R_2 R_1) \).
    The identity element is \(e = (I,0)\).
    Its Lie algebra of generators is \(\mathfrak{se}(3) = \{ (v, \omega) \in \bbR^3 \times \bbR^{3 \times 3} \mid \omega^\top + \omega = 0 \}\).
\end{definition}

For a generator \((v,\omega) \in \mathfrak{se}(3)\) we call \(v\) the \emph{translation velocity vector} and \(\omega\) the \emph{angular velocity tensor}.
The \textit{angular velocity} \(\|\omega\|\) of an angular velocity tensor \(\omega\) is defined as
\begin{equation} \label{eq:angular_velocity}
    \|\omega\| := \sqrt{\omega_1^2 + \omega_2^2 + \omega_3^2} \quad \text{where} \quad  \omega = \left( \begin{matrix}
        0 & -\omega_3 & \omega_2 \\
        \omega_3 & 0 & -\omega_1 \\
        -\omega_2 & \omega_1 & 0
    \end{matrix} \right).
\end{equation}

We define the action \(\act : \SE(3) \times \bbM_3 \to \bbM_3\) of the roto-translation group on position-orientation space by
\begin{equation} \label{eq:se3_act_m3}
    (t, R) \act (x, n) = (t + R x, R n),
\end{equation}
where \((t, R) \in \SE(3)\) and \((x,n) \in \bbM_3\).
This action naturally extends to an action \(\act : \SE(3) \times T\bbM_3 \to T\bbM_3\) (via the pushforward):
\begin{equation} \label{eq:se3_act_tm3}
        (t, R) \act (\dot x, \dot n) = (R\dot x, R\dot n),
\end{equation}
where \((\dot x, \dot n) \in T_{(x,n)} \bbM_3\).
The action also induces the action \(\act : \mathfrak{se}(3) \times \bbM_3 \to T\bbM_3\) given by 
\begin{equation}
    (v, \omega) \act (x, n) = (v + \omega x, \omega n),
\end{equation}
where \((v, \omega) \in \mathfrak{se}(3)\), and \((v + \omega x, \omega n) \in T_{(x,n)}\bbM_3\).
From here on out we will sometimes drop the group action symbol \(\act\) for conciseness.

\section{Invariant Metrics} \label{sec:invariant_metrics}

In this section we will classify all \(\SE(3)\) invariant metrics \(\cG\) on \(\bbM_3\).
We will only state the (parallelogram law abiding) norm \(\|\cdot\|_\cG\).
One can reconstruct the corresponding inner product using the standard polarization identity.

\begin{theorem}[\(\SE(3)\) Invariant Metrics on \(\bbM_3\)] \label{res:metrics}
    Let \(p=(x,n) \in \bbM_3\) and \(\dot p = (\dot x, \dot n) \in T_p\bbM_3\).
    Every $\SE(3)$ invariant Riemannian metric tensor field \(\cG\) on \(\bbM_3\) yields a norm of the form
    \begin{equation} \label{eq:se3_metric_m3}
        \|(p,\dot p)\|_\cG^2 
        = w_1 \, |\dot x \cdot n|^2 
        + w_2 \, \|\dot x \times n\|^2 
        + w_3 \, \|\dot n \|^2 
        + 2 w_4 \, \dot x \cdot \dot n 
        + 2 w_5 \, \dot x \cdot (\dot n \times n),
    \end{equation}
    with \(w_i \in \bbR\) constants that satisfy the positivity\footnote{\(\|(p,\dot p)\|\geq0\) and \(\|(p,\dot p)\| = 0 \Rightarrow \dot p=0\).} constraints \(w_1, w_2, w_3 > 0\) and \(w_2 w_3 > w_4^2 + w_5^2\).
\end{theorem} 

\begin{proof}~

    \textbf{Invariance.}
    To prove that the stated metric is indeed \(\SE(3)\) invariant we must show that \(\|g \act (p, \dot p)\|_\cG = \|(p,\dot p)\|_\cG\) for all \(g \in \SE(3)\) and \((p, \dot p) \in T\bbM_3\).
    This is quickly verified using \eqref{eq:se3_act_m3} and \eqref{eq:se3_act_tm3} and basic properties of the inner and cross product.

    \textbf{Classification.}
    Consider an arbitrary \(\SE(3)\) invariant metric \(\cH\). 
    To prove that \(\cH\) is of the form \eqref{eq:se3_metric_m3} it suffices to show that \(\cH_{p}\) coincides with \eqref{eq:se3_metric_m3} at any position-orientation \(p=(x,n)\).
    That the \(w_i\) have to be constants then follows from the invariance together with the fact that \(\SE(3)\) acts transitively on \(\bbM_3\).
    
    Let \(e_i\) be any orthonormal basis of \(\bbR^3\) with \(e_1 = n\).
    Define the basis \(f_1 = (e_1,0)\), \(f_2 = (e_2,0)\), \(f_3 = (e_3,0)\), \(f_4 = (0,e_2)\), \(f_5 = (0,e_3)\) for \(T_{p}\bbM_3\).
    One can check that this is indeed a basis of \(T_p \bbM_3\) (\Cref{def:pos_ori_space}).
    Define the components \(h_{ij} = \cH_{p}(f_i, f_j)\).
    
    The metric \(\cH_{p}\) should be invariant under \(\stab p \subset \SE(3)\) which are rotations in the plane spanned by \(e_2\) and \(e_3\) (positioned at \(x\)).
    More specifically, consider the roto-translation \(g \in \stab p\) that rotates \(e_2 \mapsto e_3\) and \(e_3 \mapsto -e_2\).
    Enforcing invariance under this specific rotation in the \(f\) basis we find the equations \(\cH_{p}(f_i, f_j) = \cH_{p}(g f_i, g f_j)\) for all \(i, j = 1, \dots, 5\). 
    Two examples are
    \begin{equation}
    \begin{split}
        &h_{23}=\cH_{p}(f_2, f_3)=\cH_{p}(g f_2, g f_3)=\cH_{p}(f_3, -f_2)=-h_{32}, \text{ and}\\
        &h_{34}=\cH_{p}(f_3, f_4)=\cH_{p}(g f_3, g f_4)=\cH_{p}(-f_2, f_5)=-h_{25}.
    \end{split}
    \end{equation}
    These equations together with the symmetry constraints \(h_{ij}=h_{ji}\) leaves us with the following component matrix
    \begin{equation} \label{eq:metric_reduced_components}
        h_{ij} = \begin{bmatrix}
            h_{11} & 0 & 0 & 0 & 0 \\
            0 & h_{22} & 0 & h_{24} & h_{25} \\
            0 & 0 & h_{22} & -h_{25} & h_{24} \\
            0 & h_{24} & -h_{25} & h_{44} & 0 \\
            0 & h_{25} & h_{24} & 0 & h_{44} \\
        \end{bmatrix}_{ij}
        = \begin{bmatrix}
            w_1 & 0 & 0 & 0 & 0 \\
            0 & w_2 & 0 & w_4 & w_5 \\
            0 & 0 & w_2 & -w_5 & w_4 \\
            0 & w_4 & -w_5 & w_3 & 0 \\
            0 & w_5 & w_4 & 0 & w_3 \\
        \end{bmatrix}_{ij},
    \end{equation}
    where we relabeled \(h_{11}, h_{22}, h_{44}, h_{24}, h_{25} \mapsto w_1, w_2, w_3, w_4, w_5\).

    Consider a tangent vector \(\dot p = (\dot x, \dot n) = c^i f_i\), that is \(\dot x = c^1 e_1 + c^2 e_2 + c^3 e_3\) and \(\dot n = c^4 e_2 + c^5 e_3\), and we remind the reader that \(n = e_1\).
    As the final step, we can check that \(\|(p, \dot p)\|_\cH^2\) is indeed of the form \eqref{eq:se3_metric_m3}:
    
    \begin{equation}
    \begin{split}
        \|(p, \dot p)\|_\cH^2
        = h_{ij} c^i c^j
        &= w_1 c^1 c^1 
         + w_2(c^2 c^2 + c^3 c^3) 
         + w_3(c^4 c^4 + c^5 c^5) \\
        &+ 2 w_4(c^2 c^4 + c^3 c^5)
         + 2 w_5(c^2 c^5 - c^3 c^4) \\
        &= w_1 ((c^1 e_1 + c^2 e_2 + c^3 e_3) \cdot e_1)^2 \\
        &+ w_2 \|(c^1 e_1 + c^2 e_2 + c^3 e_3) \times e_1\|^2 \\
        &+ w_3 \|c^4 e_2 + c^5 e_3\|^2 \\
        &+ 2 w_4 (c^1 e_1 + c^2 e_2 + c^3 e_3) \cdot (c^4 e_2 + c^5 e_3) \\
        &+ 2 w_5 (c^1 e_1 + c^2 e_2 + c^3 e_3) \cdot ((c^4 e_2 + c^5 e_3) \times e_1) \\
        &= \|(p, \dot p)\|_\cG^2.
    \end{split}
    \end{equation}

    \textbf{Positivity.}
    To prove the positivity constraints we will apply Sylvester's criterion to the matrix \eqref{eq:metric_reduced_components} that encodes the metric.
    The determinants of the upper-left minors are, respectively, \(w_1\), \(w_1 w_2\), \(w_1 w_2^2\), \(w_1 w_2 (w_2 w_3 - w_4^2 - w_5^2)\), and \(w_1 (w_2 w_3 - w_4^2 - w_5^2)^2\).
    Requiring all of these determinants to be positive and simplifying the resulting system of inequalities yields \(w_1,w_2,w_3 > 0\) and \(w_2 w_3 > w_4^2 + w_5^2\).
\end{proof}

\begin{remark}
    This result is in contradiction with \cite[Prop.14]{smets2023pde} where it is claimed that every invariant metric is of the form 
    \(\|(p,\dot p)\|^2 
    = w_1 \, |\dot x \cdot n|^2 
    + w_2 \, \|\dot x \times n\|^2 
    + w_3 \, \|\dot n \|^2.\)
    These are only the invariant metrics that are diagonal with respect to the \(f_i\) basis.
\end{remark}

\begin{remark}
    The positivity constraint \(w_2 w_3 > w_4^2 + w_5^2\) seems to be related to \(\|\dot{x} \times n\|^2 \|\dot{n} \|^2 = (\dot{x} \cdot \dot{n})^2 + (\dot{x} \cdot (\dot{n} \times n))^2\) but the exact relation is unclear to the authors.
\end{remark}

\begin{remark}
    In a similar fashion one can show that every metric that is invariant under the rigid transformation group $\E(3) = \bbR^3 \rtimes \text{O}(3)$ is of the form
    \(
        \|(p,\dot p)\|_\cG^2 
        = w_1 \, |\dot x \cdot n|^2 
        + w_2 \, \|\dot x \times n\|^2 
        + w_3 \, \|\dot n \|^2 
        + 2 w_4 \, \dot x \cdot \dot n,
    \)
    which can be obtained from \eqref{eq:se3_metric_m3} by setting \(w_5 = 0\).
    In fact, this result generalizes to \(\E(d)\) invariant metrics on \(\bbM_d\) for \emph{all} \(d \geq 2\) by replacing \(\|\dot x \times n\|\) with \(\|\dot x \wedge n \|\).
\end{remark}







\section{Mav Generator} \label{sec:mav}

In this section we describe how one can obtain the minimal angular velocity (mav) generator \(M(p_1, p_2) \in \mathfrak{se}(3)\) between two position-orientations \(p_1, p_2 \in \bbM_3\).

We will proceed as follows.
First, we start by constructing the planar roto-translation \(S = \exp(M) \in \SE(3)\), that being the exponential of the mav generator.
Second, we rewrite the planar roto-translation \(S\) as a screw displacement.
Third, we find a generator of the obtained screw displacement, this being the mav generator.
We take this approach because identifying a generator of a screw displacement is more straightforward.

Even though the planar roto-translation was already described in \cite{portegies2015new}, we provide a new geometric perspective and  construction, combined with a way to efficiently calculate the corresponding mav generator.

\begin{definition}[Planar Roto-Translation \& Mav Generator] \label{def:mav}
    Let \(p_1=(x_1,n_1)\), \(p_2=(x_2,n_2) \in \bbM_3\) be two position-orientations.
    
    The \emph{planar roto-translation} \(S(p_1,p_2) \in \SE(3)\) from \(p_1\) to \(S \act p_1 = p_2\) is the unique roto-translation with its rotation being purely in the plane made by the orientations \(n_1\) and \(n_2\).
    
    The \emph{mav generator} \(M(p_1,p_2) \in \mathfrak{se}(3)\) is the unique generator of the planar roto-translation, that is \(\exp M = S\), with minimal angular velocity \eqref{eq:angular_velocity}.
\end{definition}

Explicitly constructing the planar roto-translation \(S\) can be done as follows.
Let \(N = \text{span} \{ n_1, n_2 \}\) be the rotation plane, 
\(\theta = \arccos(n_1 \cdot n_2) \in [0, \pi]\) the rotation angle, 
\(L = (n_2 n_1^\top - n_1 n_2^\top) / \sin(\theta) \) the unit radian rotation generator in \(N\), 
\( \omega = \theta L\) the rotation generator, 
and define \(R = \exp \omega\).
Using Rodrigues' rotation formula we can simplify the exponential to \(R = \exp \omega = I + \sin(\theta) L + (1 - \cos(\theta)) L^2\).
The planar roto-translation \(S\) is then \(S : x \mapsto R (x - x_1) + x_2\).

We can also explicitly construct the mav generator \(M\).
To find it we will write the planar roto-translation as a \emph{screw displacement}.

\begin{definition}[Screw Displacement]
    A \emph{screw displacement} is a roto-trans\-lation in the form of $x \mapsto c + R (x - c) + t$ with \(R t = t\), that is the translational part \(t\) is orthogonal to the plane of rotation.
    We call \(c\) the \emph{center of rotation} of the screw displacement.
\end{definition}

Every roto-translation can be written as a screw displacement, this is known as Chasles' theorem \cite{heard2005rigid}.
More specifically, we can write the planar roto-translation as a screw displacement as follows.
Let \(\vec{x} = x_2 - x_1\) be the vector from \(x_1\) to \(x_2\), \(x_m = (x_1 + x_2)/2\) the midpoint, and decompose \(\vec{x} = \vec{x}_\parallel + \vec{x}_\perp\) into a parallel and perpendicular part w.r.t. the rotation plane \(N\).
Define \(c = x_m + \tfrac{1}{2} \cot \tfrac{\theta}{2} \, L \vec{x}_\parallel\).
The \emph{planar screw displacement} is then \(x \mapsto c + R(x - c) + \vec{x}_\perp\).
The situation is visualized in \Cref{fig:mav_screw_displacement}, and should explain geometrically how we obtained the formula for the center of rotation \(c\).

\begin{figure}
    \centering
    \begin{subfigure}{0.3\textwidth}
        \newcommand{\scale}{0.5}
        \newcommand{\viewa}{15}%
        \newcommand{\viewb}{38}%
        \newcommand{\pointsize}{2pt}%
        \input{center_of_rotation_3d.txt}%
        \caption{Overview.}
    \end{subfigure}%
    \begin{subfigure}{0.3\textwidth}
        \newcommand{\scale}{0.5}
        \newcommand{\viewa}{0}%
        \newcommand{\viewb}{90}%
        \newcommand{\pointsize}{2pt}%
        \input{center_of_rotation_3d.txt}%
        \caption{Top view.}
    \end{subfigure}
    \begin{subfigure}{0.3\textwidth}
        \newcommand{\scale}{0.6}
        \newcommand{\viewa}{26.5}%
        \newcommand{\viewb}{0}%
        \newcommand{\pointsize}{2pt}%
        \input{center_of_rotation_3d.txt}%
        \caption{Front view.}
    \end{subfigure}
    \caption{Diagram of the center of rotation \(c\) of the planar screw displacement between two position-orientations \(p_1 = (x_1,n_1)\), \(p_2 = (x_2, n_2) \in \bbM_3\). }
    \label{fig:mav_screw_displacement}
\end{figure}
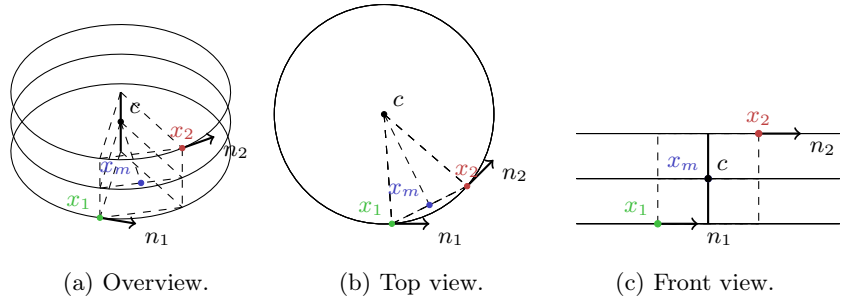

Finding a generator for a screw displacement $x \mapsto c + R (x - c) + t$ is straightforward.
Namely, a generator is $x \mapsto \omega (x - c) + v$, where $\omega$ is a generator of $R$ in $\SO(3)$, and \(v = t\) is the generator of \(t\) in \(\bbR^3\).
Particularly, the mav generator \(s\) for the planar screw displacement \(S\) is \(x \mapsto \omega(x - c) + \vec{x}_\perp\) where \(\omega = \theta L\).

So, all in all, we have the following explicit form for the mav generator.

\begin{proposition}[Explicit Mav Generator] \label{res:mav_generator}
    The \emph{mav generator} \(M(p_1,p_2) \in \mathfrak{se}(3)\) between two position-orientations \(p_1=(x_1,n_1)\), \(p_2=(x_2,n_2) \in \bbM_3\) is
    \begin{equation}
        M = (- \omega c + v, \omega),
    \end{equation}
    where \(\omega = \theta L\), \(c = x_m + \tfrac{1}{2} \cot \tfrac{\theta}{2} \, L \vec{x}_\parallel\), \(v = \vec{x}_\perp\), \(\theta = \arccos(n_1 \cdot n_2)\), \(L = (n_2 n_1^\top - n_1 n_2^\top) / \sin(\theta)\), \(x_m = (x_1 + x_2)/2\), and \(\vec{x} = x_2 - x_1\).
\end{proposition}

\section{Experiments} \label{sec:experiments}

The standard PONITA architecture builds upon 3 scalar \(\SE(3)\) invariants (collected in a column vector), as designed in Bekkers et al. \cite[eq.9]{bekkers2024fast}:
\begin{equation} \label{eq:bekkers_invariants}
\begin{split}
    \iota(p_1, p_2) &= \left( \begin{matrix}
        (x_2 - x_1) \cdot n_1\\
        \|(x_2 - x_1) - \iota_1(p_1, p_2) n_1 \|\\
        \arccos(n_1 \cdot n_2)
    \end{matrix} \right), \\
\end{split}
\end{equation}
where \(p_1 = (x_1, n_1)\), \(p_2 = (x_2, n_2) \in \bbM_3\).
We compare this architecture against one where these 3 invariants are \emph{completely} replaced by the mav distance \(\mu_\cG\):
\begin{equation} \label{eq:mav_distance_experi}
    \mu_\cG(p_1,p_2) = \| M(p_1, p_2) \act p_1 \|_\cG,
\end{equation}
where \(\cG\) is the \(\SE(3)\) invariant metric on \(\bbM_3\) determined by the learnable metric parameters \(w_1, \dots, w_5\) as described in \Cref{res:metrics}, and \(M(p_1,p_2)\) the mav generator as described in \Cref{res:mav_generator}.
We investigate if this replacement has a positive impact on the accuracy of PONITA.

Through a reparameterization of the metric parameters \(w_1, \dots, w_5\) one can enforce the (semi-)positivity constraints \(w_1, w_2, w_3 \geq 0\) and \(w_2 w_3 \geq w_4^2 + w_5^2\).
For example, one can define \(w_1 = a_1^2\), \(w_2 = a_2^2\), \(w_3 = a_3^2\), \(w_4 = a_4 \delta \), and \(w_5 = a_5 \delta \), where \(\delta = 2 a_2 a_3 / (1 + a_4^2 + a_5^2)\).
However, in practice, enforcing these constraints yielded no noticeable benefits. 
We therefore left the metric parameters unconstrained, allowing the mav ``distance'' to take on negative values.
This modification preserves the \(\SE(3)\) invariance of the mav distance.

We performed some experiments on the QM9 dataset \cite{ramakrishnan2014quantum,ruddigkeit2012enumeration}, specifically predicting properties of various molecules (134k stable small organic molecules).
We choose to discretize with \(16\) orientations, use \(6\) layers, \(128\) dimensional features, and train for \(800\) epochs.
All other model settings and hyperparameters are kept the same as in \cite[App.E.2]{bekkers2024fast}.
We report the mean absolute error on the test set with the model that did best on the validation set during training.
The results can be found in \Cref{tab:ponita_qm9}.

\definecolor{good}{RGB}{54, 176, 60}
\definecolor{bad}{RGB}{176, 74, 110}

We see that the use of the mav distance within PONITA has a positive impact on accuracy on 7 of the 12 targets we experimented on, with an average improvement of \(\color{good}-4.0\%\).

\begin{table}
    \centering
    \def\arraystretch{1.2}
    \setlength{\tabcolsep}{0.5em} 
    \begin{tabular}{ll|rrr}
        Target & Unit & Bekkers et al. 
        \eqref{eq:bekkers_invariants} & Mav Distance \eqref{eq:mav_distance_experi} & Difference \% \\
        \hline
        \(\mu\) & D & 0.0195 & 0.0181 & \(\color{good}-07.2\)\\
        \(\alpha\) & \(a_0^3\) & 0.0556 & 0.0540 & \(\color{good}-02.9 \)\\
        \(\varepsilon_\text{homo}\) & eV & 0.0225 & 0.0229 & \(\color{bad}+01.8 \)\\
        \(\varepsilon_\text{lumo}\) & eV & 0.0205 & 0.0207 & \(\color{bad}+01.0 \)\\
        \(\Delta \varepsilon\) & eV & 0.0414 & 0.0431 & \(\color{bad}+04.0 \)\\
        \(\langle R^2 \rangle\) & \(a_0^2\) & 0.4160 & 0.4942 & \(\color{bad}+18.8\)\\
        ZPVE & meV & 1.5647 & 1.5613 & \(\color{good}-00.2\)\\
        \(U_0\) & eV & 0.9920 & 0.7047 &  \(\color{good}-28.9\)\\
        \(U\) & eV & 1.3593 & 1.0947 & \(\color{good}-19.5\)\\
        \(H\) & eV & 1.0204 & 1.0856 & \(\color{bad}+06.4\)\\
        \(G\) & eV & 1.1856 & 0.9691 & \(\color{good}-18.3\)\\
        \(c_v\) & \(\frac{\mathrm{cal}}{\mathrm{mol} \, \mathrm{K}}\) & 0.0291 & 0.0283 & \(\color{good}-02.8\)\\
    \end{tabular}
    \caption{
        PONITA trained to predict chemical properties of various molecules (QM9 dataset).
        Mean absolute error on the test set is reported (lower is better). 
    }
    \label{tab:ponita_qm9}
\end{table}

\section{Conclusion} \label{sec:conclusion}

In \Cref{sec:invariant_metrics} we classified and parametrized all \(\SE(3)\) invariant Riemannian metric tensor fields \(\cG\) on \(\bbM_3\).
In \Cref{sec:mav} we explicitly constructed the mav generator \(M(p_1, p_2) \in \mathfrak{se}(3)\) between two position-orientations \(p_1, p_2 \in \bbM_3\).
These two results allowed us to use the mav distance \(\mu_\cG(p_1, p_2)\) as a trainable invariant in the PONITA architecture.

In \Cref{sec:experiments} we performed an experimental comparison between the mav distance \eqref{eq:mav_distance_experi} and the invariants proposed in Bekkers et al. \cite{bekkers2024fast}, as defined in \eqref{eq:bekkers_invariants}. 
We did this by training the PONITA architecture to predict chemical properties of various molecules (QM9 dataset \cite{ramakrishnan2014quantum,ruddigkeit2012enumeration}).
We observe marginal improvements in accuracy when using the mav distance.

Even though the improvements are marginal, our theoretical results can still find application in the processing of medical images, as evidenced by the works \cite{duits2011left,portegies2015new,smets2023pde,smets2021total}.

\textbf{Availability of Code.} 
All code can be found at 
{\if\anonymous0
\url{https://gitlab.com/gijsbel/ponita_invariants}.
\else
\url{https://anonymized.url}.
\fi}

{\if\anonymous0
\textbf{Acknowledgements.} 
The European Union is gratefully acknowledged for financial support through project REMODEL (Horizon Europe, MSCA-SE, 101131557 \url{https://doi.org/10.3030/101131557}).
The Dutch Research Council (NWO) is gratefully acknowledged for financial support via VIC.202.031 (\url{https://www.nwo.nl/en/projects/vic202031}).
\fi}
We thank Bekkers et al. \cite{bekkers2024fast} for their publicly available PONITA architecture \url{https://github.com/ebekkers/ponita}.

\bibliographystyle{splncs04}
\bibliography{bibliography.bib}

\end{document}

%% file: center_of_rotation_3d.txt
\begin{tikzpicture}[scale=\scale, 3d view={\viewa}{\viewb}]
    \definecolor{x1color}{rgb}{0.25,0.75,0.25}%
    \definecolor{x2color}{rgb}{0.75,0.25,0.25}%
    \definecolor{xmcolor}{rgb}{0.25,0.25,0.75}%

    \newcommand{\rad}{2.92}
    
    \coordinate (x1) at (0.0, 0.0, 0.0);
    \coordinate (x2) at (2.0, 1.0, 2.0);
    \coordinate (n1) at (1.0, 0.0, 0.0);
    \coordinate (n2) at (0.7, 0.7, 0.0);

    \coordinate (n3) at (0.0, 0.0, 1.0);
    \coordinate (xm) at ($(x1)!0.5!(x2)$);
    \coordinate (xd) at ($(x2)-(x1)$);
    \coordinate (xp) at (-1.0,2.0,0.0);
    
    \coordinate (cm) at ($(xm) + 2.4142*0.5*(xp) $);
    \coordinate (cb) at ($(cm) - (0,0,1)$);
    \coordinate (ct) at ($(cm) + (0,0,1)$);

    \coordinate (x1t) at ($(x1) + (0,0,2)$);
    \coordinate (x1m) at ($(x1) + (0,0,1)$);
    \coordinate (x2b) at ($(x2) - (0,0,2)$);
    \coordinate (x2m) at ($(x2) - (0,0,1)$);
    
    
    \draw[thick,->] (x1) -- ++(n1);
    \draw[thick,->] (x2) -- ++(n2);

    \draw[dashed] (cb) -- (x1);
    \draw[dashed] (cb) -- (x2b);
    
    \draw[dashed] (ct) -- (x2);
    \draw[dashed] (ct) -- (x1t);
    
    \draw[dashed] (cm) -- (xm);
    \draw[dashed] (cm) -- (x2m);
    \draw[dashed] (cm) -- (x1m);
    \draw[dashed] (x1m) -- (x2m);
    \draw[dashed] (x1t) -- (x2);
    \draw[dashed] (x2b) -- (x1);
    
    \draw[thick] (cb) -- (ct);
    \draw[dashed] (x2) -- (x2b);
    \draw[dashed] (x1) -- (x1t);

    \draw[] (cb) circle (\rad);
    \draw[] (cm) circle (\rad);
    \draw[] (ct) circle (\rad);

    \filldraw[x1color] (x1) circle (\pointsize) node [x1color, anchor=south east] {$x_1$};
    \filldraw[x2color] (x2) circle (\pointsize) node [x2color, anchor=south] {$x_2$};
    \filldraw[xmcolor] (xm) circle (\pointsize) node [xmcolor, anchor=south east] {$x_m$};
    \filldraw[black] (cm) circle (\pointsize)  node [anchor=south west] {$c$};

    \draw (x1) ++ (n1) node [anchor=north west] {$n_1$};
    \draw (x2) ++ (n2) coordinate (bla2) node [anchor=north west] {$n_2$};


\end{tikzpicture}

%% file: GSI25-070.bbl
\begin{thebibliography}{10}
\providecommand{\url}[1]{\texttt{#1}}
\providecommand{\urlprefix}{URL }
\providecommand{\doi}[1]{https://doi.org/#1}

\bibitem{bekkers2024fast}
Bekkers, E.J., Vadgama, S., Hesselink, R., {van der Linden}, P.A., Romero, D.W.: Fast, expressive {SE}(3) equivariant networks through weight-sharing in position-orientation space. In: ICLR (2024), \url{https://openreview.net/forum?id=dPHLbUqGbr}

\bibitem{christensen2020role}
Christensen, A.S., {von Lilienfeld}, O.A.: On the role of gradients for machine learning of molecular energies and forces. Mach. Learn.: Sci. Techn.  \textbf{1}(4) (2020). \doi{10.1088/2632-2153/abba6f}

\bibitem{duits2018optimal}
Duits, R., Meesters, S.P.L., Mirebeau, J.M., Portegies, J.M.: Optimal paths for variants of the 2{D} and 3{D} {Reeds--Shepp} car. JMIV  \textbf{60}(6),  816--848 (2018), \url{https://doi.org/10.1007/s10851-018-0795-z}

\bibitem{duits2011left}
Duits, R., Franken, E.: Left-invariant diffusions on the space of positions and orientations. IJCV  \textbf{92}(3),  231--264 (2011). \doi{10.1007/s11263-010-0332-z}

\bibitem{heard2005rigid}
Heard, W.B.: Rigid Body Mechanics: Mathematics, Physics and Applications (2005). \doi{10.1002/9783527618811}

\bibitem{portegies2015new}
Portegies, J., Sanguinetti, G., Meesters, S., Duits, R.: New approximation of a scale space kernel on {SE}(3). In: SSVM. pp. 40--52 (2015). \doi{10.1007/978-3-319-18461-6_4}

\bibitem{ramakrishnan2014quantum}
Ramakrishnan, R., Dral, P.O., Rupp, M., {von Lilienfeld}, O.A.: Quantum chemistry structures and properties of 134 kilo molecules. Scientific Data  \textbf{1}(1) (2014). \doi{10.1038/sdata.2014.22}

\bibitem{ruddigkeit2012enumeration}
Ruddigkeit, L., {van Deursen}, R., Blum, L.C., Reymond, J.L.: Enumeration of 166 billion organic small molecules in the chemical universe database {GDB-17}. Journal of Chemical Information and Modeling  \textbf{52}(11),  2864--2875 (2012). \doi{10.1021/ci300415d}

\bibitem{smets2023pde}
Smets, B.M.N., Portegies, J., Bekkers, E.J., Duits, R.: {PDE}-based group equivariant convolutional neural networks. JMIV  \textbf{65}(1),  209--239 (2023). \doi{10.1007/s10851-022-01114-x}

\bibitem{smets2021total}
Smets, B.M.N., Portegies, J., St-Onge, E., Duits, R.: Total variation and mean curvature {PDE}s on the homogeneous space of positions and orientations. JMIV  \textbf{63}(2),  237--262 (2021). \doi{10.1007/s10851-020-00991-4}

\end{thebibliography}
